\documentclass{arx}
\usepackage{amsmath}
  \usepackage{paralist}
  \usepackage{graphics} 
  \usepackage{epsfig} 
\usepackage{graphicx}  \usepackage{epstopdf}
 \usepackage[colorlinks=true]{hyperref}
\hypersetup{urlcolor=blue, citecolor=red}

 \usepackage[pagewise]{lineno}
 
\newtheorem{theorem}{Theorem}[section]
\newtheorem{corollary}{Corollary}

\newtheorem{lemma}[theorem]{Lemma}
\newtheorem{proposition}{Proposition}

\theoremstyle{definition}
\newtheorem{definition}[theorem]{Definition}
\newtheorem{remark}{Remark}

\newcommand{\ep}{\varepsilon}
\newcommand{\eps}[1]{{#1}_{\varepsilon}}

\newcommand{\R}{{\Bbb R}}

\newcommand{\C}{{\Bbb C}}

\title[Pushed wavefronts for a monostable non-monotone delayed model]{On pushed wavefronts of monostable equation with unimodal delayed reaction}

\author[Has\'ik Kopfov\'a  N\'ab\v elkov\'a and Trofimchuk]{}


\email{Karel.Hasik@math.slu.cz}
\email{Jana.Kopfova@math.slu.cz}
\email{petra.nabelkova@math.slu.cz}
 \email{trofimch@inst-mat.utalca.cl}

\begin{document}
\maketitle

\centerline{\scshape Karel Has\'ik, Jana Kopfov\'a, Petra N\'ab\v elkov\'a }
\medskip
{\footnotesize
 \centerline{Mathematical Institute, Silesian University, 746 01 Opava, Czech Republic}
} 

\medskip

\centerline{\scshape Sergei Trofimchuk$^*$
}
\medskip
{\footnotesize
 \centerline{Instituto de Matem\'atica, Universidad de Talca, Casilla 747,
Talca, Chile}
}

\bigskip


\begin{abstract}
 We study the Mackey-Glass type monostable delayed reaction-diffusion equation with a unimodal birth function $g(u)$. 
This model, designed to describe evolution of single species populations,  is considered here in the presence of the weak Allee effect 
($g(u_0)>g'(0)u_0$ for some $u_0>0$).  We focus our  attention on the existence of  slow monotonic traveling fronts  to the equation: under given assumptions, 
this problem seems to be rather  difficult since the  usual positivity and monotonicity arguments are not effective.  First, we solve the front existence problem
for small delays, $h \in [0,h_p]$, where $h_p$  (given by an explicit formula) is optimal in a certain sense.  Then we take a representative piece-wise linear unimodal
birth function making possible explicit computation of traveling fronts. In this case, we find out that a) increase of delay can 
destroy asymptotically stable pushed fronts; b) the set of  all admissible wavefront speeds  has usual structure of a semi-infinite interval $[c_*, +\infty)$;  c) for each $h\geq 0$, the pushed wavefront is unique (if it exists); d) pushed wave can oscillate slowly around the positive equilibrium for sufficiently large delays. 

\end{abstract}

\vspace{5mm}

\section{Introduction}
\noindent
In this work, we consider the Mackey-Glass type delayed reaction-diffusion equation
 \begin{equation}\label{MGD}
u_t(t,x) = u_{xx}(t,x)  - u(t,x) + g(u(t-h,x)), \quad u \geq 0,\ (t,x) \in \R^2, \ h \geq 0, 
\end{equation}
widely  used to model evolution of single species populations  (see \cite{BY} for details and further references). 
In such an ecological interpretation of  (\ref{MGD}),   $g: \R_{+} \to \R_{+} $ is called a birth function, and if we further assume that $g'(0)\geq 1$ and that the equation $g(x)= x$ has 
just two  solutions, $0$ and $\kappa >0$,   this equation belongs to the class of the so-called  monostable  population models.  Clearly, monostable equation (\ref{MGD}) has exactly 
two equilibria, $u=0$ and $u=\kappa$.  It  can  also have positive wave solutions $u(t,x)=\phi(x+ct), \ \phi(-\infty)=0, \ \liminf_{t\to +\infty}\phi(t) >0$, corresponding to  the transition regimes 
between these two equilibria. These solutions called semi-wavefronts (or wavefronts if, in addition,  $\phi(+\infty)=\kappa$)  characterize spatial propagation of the species  and thus  are quite significant from the ecological point of view.  The description of  the set $\mathcal C(h)$ of all possible velocities of semi-wavefronts (for each fixed $h>0$) is  of evident practical importance, this problem  is also one of fundamental  interest in the traveling wave  theory. It follows from (\ref{MGD}) that $\mathcal C(h)$ depends solely on the properties of $g$ (for each fixed $h\geq 0$).  

In 
particular, it is known that $\mathcal C(h)$  is a  closed unbounded interval, $\mathcal C(h) =[c_*(h),+\infty)$ whenever $C^1$-smooth $g$ is either monotone on $[0,\kappa]$ or it satisfies 
the following two (sub-tangency and unimodality) conditions:
\begin{equation}\label{subt}
g(x) \leq g'(0)x, \qquad x \geq 0, 
\end{equation}

\vspace{2mm}

\noindent  {\rm \bf(U)}  there exists $0< \theta <\kappa$  such that $g$ increases  on the interval $[0,\theta)$ and is decreasing   otherwise; we are assuming also that $g'(0) >1.$ 

\vspace{2mm}
\noindent Then real number $c_*(h)$ belonging to the boundary of $\mathcal C(h)=[c_*(h),+\infty)$, is called the minimal (or critical) speed of propagation, it has a special status in the theory and applications. In particular,  occupation of a new environment by invasive species  is realized precisely with the minimal speed.

It should be noted that the assumption {\rm \bf(U)}  of unimodal shape of $g$ is quite natural from the ecological point of view  \cite{BY} and  can barely be considered as an essential  
limitation. The situation when  $g$ is increasing on $[0,\kappa]$ is considerably  simpler for mathematical analysis due to the availability of comparison techniques. Owing to the previous studies (cf. \cite{LZii,tpt,WNH}) properties of waves in  the monotone model  (\ref{MGD}) are well understood so  we will assume in the sequel  that $g$ is non-monotone and has  unimodal  shape.  In such a case,  the topological structure of $\mathcal C(h)$  can  be  potentially more complex  when, in addition, $g$ does not satisfy inequality (\ref{subt}) (recall that  (\ref{subt}) is an important ecological restriction excluding the weak Allee effect\footnote{On the other hand,  (\ref{subt}) is essential for the analysis of  (\ref{MGD}), allowing to invoke positivity arguments based on non-negativity of the function  $g'(0)x-g(x)$.}).
In \cite[question (iii), p. 107]{ai},  S. Ai   posed a question  about the existence of the minimal speed in its usual meaning of a unique boundary point of $\mathcal C(h)$,  for  monotone fronts to  monostable non-monotone delayed  equations. Theorem 1.3 in  \cite{ai}  answers partially this question for some special  models with   {\it distributed} delays which admit transformation of associated  delayed profile equations into the systems of ordinary differential equations.  Fenichel's geometric singular perturbation theory was instrumental  in proving this result in  \cite{ai}
(which, however,  does not apply to  equations with {\it discrete} delays).

Consequently,  the delayed equation  (\ref{MGD}) with the unimodal birth function $g$ which does not satisfy the sub-tangency 
assumption (\ref{subt}) is both  an interesting and a challenging object to study. Even such starting point for the research as the question about connectedness of  the set $\mathcal C(h)$ for $h>0$ should  still be answered.    In view of our previous discussion,  we will say that the propagation speed $c$ (and the associated traveling wave) is critical if it belongs to the boundary of the set $\mathcal C(h)$.  It is natural to expect that the critical wavefronts correspond to  the key transition regimes in the model, and if the set $\mathcal C(h)$ is not connected,  equation  (\ref{MGD}) can have multiple special  modes of propagation.  The problem concerning the uniqueness of the critical semi-wavefronts for equation   (\ref{MGD}) seems to be very difficult in  both monostable and bistable cases, cf. \cite{ADG,ai}, 
 however, one can also expect that it can be solved at least in the case of small delays (in the spirit of the proverb "small delays are harmless").  
In this paper, we are presenting  the first result in this direction. In fact, it is the optimal one whenever we are concerned with the monotone wavefronts: 

\begin{theorem} \label{main1}
Suppose that $g$ satisfies   {\rm \bf(U)}, $\min_{u\in[0,\kappa]}g'(u)= g'(\kappa) < 0,$  and
\begin{equation}\label{gcos}
\left|g'(u)- g'(0)\right| \leq Au^\gamma, \quad u\in  [0,\delta],
\end{equation}
for some $A>0, \delta \in (0,\theta), \gamma \in (0,1]$. 
Define $h_*>0$ as the unique real solution of the equation 
$
1=|g'(\kappa)|he^{h+1}.
$
Then  for each fixed $h \in [0,h_*]$ there is a positive 
number $c_*=c_*(h)$ (called the minimal speed of propagation) such that  equation (\ref{MGD}): (a)  possesses  a unique monotone wavefront $u(t,x)= \phi_c( x+ct)$ for every $c \geq c_*$; (b) does not have  any semi-wavefront propagating at the velocity $c < c_*$.  

The above result concerning the existence of $c_*(h)$ satisfying both requirements (a) and (b) fails to hold if $h > h_*$. 
\end{theorem}
Theorem \ref{main1} shows that the inclusion of `small' delays into model  (\ref{MGD}) does not change the  usual structure of an unbounded interval $[c_*,+\infty)$ of the set $\mathcal C$ of all admissible speeds for semi-wavefronts. Importantly, the above result presents  simple formula for the exact upper bound $h_*$ for the size of  the `small' delay (observe that  $h_*\to +\infty$ if $g'(\kappa)\to 0^-$). The existing literature on the subject presents various perturbation techniques to  treat the case of  small delays.  Specifically, here we would like to mention the Wu and Zou method  from \cite{wz} and the Ou and Wu approach in \cite{OW}. The aforementioned works show that the existence of the wavefront for non-delayed equation (\ref{MGD})  propagating at  speed $c > c_*(0)$ implies, under rather weak shape conditions on $g$, the existence of some positive $h_0(c)$ such that this wavefront persists for all $h \in [0, h_0(c)]$.   Nevertheless, these results do not allow to establish  the connectedness of  the set $\mathcal C(h)$ even for small $h$. To have a better idea of what  $\mathcal C(h)$  may look like,  we study in this paper an explicitly solvable `toy' model with  piece-wise linear (but discontinuous) unimodal birth function $g$ shown on Figure \ref{GeomVyznam}. Note that condition (\ref{subt}) in this case reads as $k \geq 3$. 
\begin{figure}[h] 
\vspace{-20mm} 
\begin{center}
\scalebox{0.32}{\includegraphics{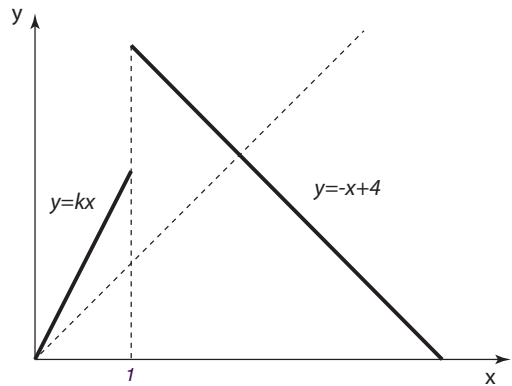}}
\vspace{-23mm} 
\caption{Toy model:  piece-wise linear birth function $g$. }
\label{GeomVyznam}
\end{center}
\end{figure}

As several previous works show (e.g. see \cite{KH,NPTT,TVN}), such a kind  of nonlinear birth functions $g$ allows to detect all essential geometric features of traveling waves that appear  in the unimodal models.   In Section \ref{Stoy}, we show that  for each $k\in (1,3)$ all traveling fronts to  equation (\ref{MGD}) considered with $g$ given on Figure \ref{GeomVyznam} can be determined in an  explicit way.  This leads to the following conclusion confirming all results of Theorem \ref{main1} (as well as  of Theorem \ref{pr1} below)  and suggesting that the simple topological  structure $\mathcal C(h)=[c_*(h),+\infty)$ of the set of all admissible semi-wavefront speeds  could also hold  for unimodal equation (\ref{MGD}). 

\noindent Let  $c=c_\#(h)$ be the unique positive number for which   the characteristic equation 
\begin{equation}\label {che}
\chi(z,c):= z^2 -cz - 1 + g'(0)e^{-zch}=0.
\end{equation}
has a double positive root (so that $c  > c_\#$ implies that (\ref{che}) 
has exactly two real solutions $0 < \lambda_2 < \lambda_1, \ \lambda_j = \lambda_j(c)$).
\begin{theorem} \label{main2} 
Suppose that $k\in (1,3)$  and  take $g$ as on Figure \ref{GeomVyznam}.  Then there exists a continuous decreasing function $c_*:[0,+\infty)\to (0,+\infty)$ such that  $c_*(h)$ is the minimal speed of propagation in the sense that equation (\ref{MGD}) has a wavefront solution 
propagating with the speed $c$  if and only if $c \geq c_*(h)$.  Furthermore, if $k \in (1,5/3)$, then there is some maximal $h_p=h_p(k) \in (0,+\infty]$ such that $c_*(h)>c_\#(h)$ for all 
$h \in [0,h_p)$. Finally, each wavefront is unique (up to translation) and for each fixed $h \geq 0$ equation (\ref{MGD}) has at most one non-linearly determined wavefront (i.e. wavefront with asymptotic representation (\ref{afecp}) given in Theorem \ref{pr1} below). 
 \end{theorem}
 Some numerical and geometrical evidences suggest that for $k$ close to 5/3,  $h_p(k)$ is finite, but if  $k$ goes closer to 1  then $h_p(k)=+\infty$ (compare the cases $k=1.5$ and $k =1.2$ in Section \ref{Stoy}). In other words, if `non-subtangency' of the birth function $g(u)$ at $0$ is relatively strong then, independently on the size of the delay, all minimal wavefronts are non-linearly determined. However, when `non-subtangency' of $g(u)$ at $0$ is relatively weak (in our toy model this surely happens  if $k \geq 1.5$), then all minimal wavefronts become linearly determined once the delay surpasses the critical value $h_p$.  This change is important for the dynamics of (\ref{MGD}) because the non-linearly determined wavefronts have  better stability properties. See also \cite{EP} for other arguments. 
 
Theorems \ref{main1} and \ref{main2} raise the question of whether the minimal speed $c_*(h)$ can be calculated explicitly from (\ref{MGD}) for each fixed $g$ and given $h\geq 0$.  It is well known that   if we assume   (\ref{subt}) then  $c_*(h)=c_\#(h)$. Without (\ref{subt}), the computation of $c_*(h)$ can be regarded  as a very difficult task even for non-delayed models \cite{BD,GK,Xin}.  It is known that $c_*(h) \in [c_\#(h),c^*(h)]$, where $c=c^*(h)$ is the unique positive number for which  
the  equation
$$
z^2 -cz - 1 + g'_+e^{-zch}=0, \quad g'_+:= \sup_{x \geq 0} g(x)/x,
$$
has a double positive root. 
As we can see, in general,  $c_*(h)$ depends on the whole nonlinearity $g$ and not only on the value of its derivative at $0$.  The critical wavefront 
$u(t,x)= \phi_*( x+c_*(h)t)$  is called pushed if  $c_*(h)> c_\#(h)$. Previous studies of monotone model (\ref{MGD})  showed that pushed wavefronts have better stability properties in comparison with 
non-critical waves, cf. \cite{WNH}.  Particularly, this is due to the fast exponential decay at $-\infty$ of the profile $\phi_*(t)$.  Our next result shows that the latter characteristic property of 
 pushed wavefronts is also valid if the delay is relatively small: 
\begin{theorem} \label{pr1} Assume  all the condition of Theorem \ref{main1}  and take some $c > c_*(h)$ for $h \in [0,h_*]$. Then   the following asymptotic representation is valid (for an appropriate $s_0$ and some $\sigma >0$):  
\begin{equation}\label {afep}
(\phi_c, \phi_c')(t+s_0,c)= e^{\lambda_2 t}(1, \lambda_2) + O(e^{(\lambda_2+ \sigma) t}), \ t \to -\infty. 
\end{equation}
If, in addition, $g \in C^{1,\gamma}[0,\kappa]$,  $c=c_*(h)> c_\#(h)$  and $h \in [0,h_*]$, then  
\begin{equation}\label {afecp}
(\phi_c, \phi_c')(t+s_0,c)=e^{\lambda_1 t}(1, \lambda_1) + O(e^{(\lambda_1+\sigma) t}), \ t \to -\infty. 
\end{equation}
\end{theorem}

\vspace{2mm}

\noindent Theorem \ref{pr1} refines the following well-known statement concerning the asymptotic representations of wavefront's profile 
at $-\infty$: 

\vspace{2mm}

\begin{proposition} \label{pr1p} Suppose that $g$ satisfies   ${\rm \bf(U)}$  and (\ref{gcos}) and  let  $u= \phi( x+ct)$ be a semi-wavefront for  (\ref{MGD}).  If, in addition,  $c >  c_\#(h)$,
 then   the following asymptotic representation is valid (for an appropriate $s_0, \ j \in \{1,2\}$ and some $\sigma >0$):  
\begin{equation}\label {afe}
(\phi, \phi')(t+s_0,c)= e^{\lambda_j t}(1, \lambda_j) + O(e^{(\lambda_j+ \sigma) t}), \ t \to -\infty. 
\end{equation}
If $c=c_\#(h)$, then there are some nonnegative $A, B$ such that $A+B>0$ and 
\begin{equation}\label {afec}
(\phi, \phi')(t+s_0,c)=(-At+B)e^{\lambda_j t}(1, \lambda_j) + O(e^{\lambda_j t}), \ t \to -\infty. 
\end{equation}
\end{proposition}

The structure of the remainder of this paper is as follows. Section \ref{fund} introduces suitable definitions of fundamental solutions for two linear  integro-differential operators and then analyzes their properties and  some relations existing  between them. The analysis in this  section offers an additional insight into  the 
properties of fundamental solutions earlier established  in \cite{TpT} by applying alternative, more technical approach. The use of  fundamental solutions as well as  `base functions' from \cite{tpt} are the key elements in the proof of Theorem \ref{main1} given in Section \ref{Sp1}. Next, Section \ref{poofs} contains short proofs of Proposition \ref{pr1p} and Theorem \ref{pr1}: we note that the proof of the latter theorem was substantially shortened due to studies realized in  \cite{LMS} (compare with the proof of Theorem 1.4 in \cite{tpt}). Such a simplification, however, required additional $C^{1,\gamma}$-smoothness property of $g(u)$.  Finally, in Section \ref{Stoy} we present detailed analytical and numerical studies of a `toy' model and prove Theorem \ref{main2}. 

 In Appendix, where the characteristic function of the variational equation at the positive steady state is analyzed, we further improve  some results established in  \cite[Lemma 1.1]{LMS}.  The obtained improvement is used in the next section. 

\section{A convolution factorization of the fundamental solution}\label{fund} 
Suppose that $u(x,t)= \phi(x+ct)$ is a wavefront solution of equation (\ref{MGD}). Then its profile $\phi$ satisfies the boundary value problem 
\begin{equation}\label{prfe}
\phi''(t) - c\phi'(t) - \phi(t) + g(\phi(t-ch)) =0, \quad \phi(-\infty)=0, \ \phi(+\infty) = \kappa. 
\end{equation}
By linearizing the above differential equation around the positive equilibrium $\kappa$, we obtain 
$$
y''(t) - cy'(t) - y(t) + g'(\kappa)y(t-ch) =0. 
$$
Considering exponential solutions $y(t) =e^{zt}$ of the latter equation,  we find that $z$ should satisfy 
$
\chi_\kappa(z) =  0, 
$
where the characteristic function $\chi_\kappa$ is given by 
$$
\chi_\kappa(z) =  z^2-cz -1 +g'(\kappa)e^{-zch}. 
$$
We will analyze the situation when $\chi_\kappa$ has exactly three real zeros, one positive and two negative (counting multiplicity), $\mu_3 \leq \mu_2 <0 < \mu_1$.  In such a  case, every complex zero $\mu_j$ of $\chi_\kappa$ is simple \cite[Lemma A.2]{TVN} and has its real part $\Re \mu_j < \mu_2$  \cite[Lemma 1.1]{LMS}.  Importantly, the latter estimate can be improved: in Appendix, we  show that actually $\Re \mu_j < \mu_3$ for each complex zero $\mu_j$ of $\chi_\kappa$. 

The set ${\mathcal D}_\kappa$ of all points
$(h,c) \in \R^2_+$ for which $\chi_\kappa$ has  three real zeros was described in \cite[Lemma 1.1 and Theorem 2.3]{LMS}: 
$$
{\mathcal D}_\kappa = [0,h_*] \times \R_+ \cup \{(h,c) \in \R^2_+: c \leq c_\kappa(h),\  h > h_*\}, 
$$
where $h_*>0$ was defined in Theorem  \ref{main1} and $c_\kappa: (h_*, +\infty) \to (0,+\infty),$  $c_\kappa(h_*+)=+\infty,$ $c_\kappa(+\infty)=0,$ is a decreasing smooth function
implicitly defined by 
\begin{equation}\label{ck+}
\hspace{-7mm}
\frac{2+\sqrt{c^4h^2+4c^2h^2+4}}{ec^2h^2|g'(\kappa)|}=\exp\left(\frac{\sqrt{c^4h^2+4c^2h^2+4}-c^2h}{2}\right), \ h > h_*.
\end{equation}
For each $(h,c) \in {\mathcal D}_\kappa$ we introduce the following integro-differential operators:
$$
(Dy)(t)=y''(t) - cy'(t) - y(t) + g'(\kappa)y(t-ch), 
$$
$$
(D_1y)(t)=y'(t) - \mu_2 y(t), 
$$
 $$
(D_2y)(t)=y'(t) - (c-\mu_2)y(t) - g'(\kappa)e^{-ch\mu_2}\int_{-ch}^0e^{-\mu_2 s}y(t+s)ds.
$$
\begin{lemma} \label{DDD} The operators $D_1$ and $D_2$ commute and 
$
Dy= D_1D_2y=D_2D_1y$ \ for every $y \in C^2(\R,\R)$.  
\end{lemma}
\begin{proof} By a straightforward computation and integration by parts we obtain 
$$
D_1D_2y = y''(t) - (c-\mu_2)y'(t) - g'(\kappa)e^{-ch\mu_2}\int_{-ch}^0e^{-\mu_2 s}y'(t+s)ds $$
$$- \mu_2\left(y'(t) - (c-\mu_2)y(t) - g'(\kappa)e^{-ch\mu_2}\int_{-ch}^0e^{-\mu_2 s}y(t+s)ds\right)= 
$$
$$
y''-cy'+(\mu_2 c -\mu_2^2-g'(\kappa)e^{-ch\mu_2})y(t)+g'(\kappa)y(t-ch)= Dy. 
$$
Similarly, 
$$
D_2D_1y = D_2(y'-\mu_2 y)= (y''-\mu_2 y') - (c-\mu_2)(y'-\mu_2 y)$$
$$-g'(\kappa)e^{-ch\mu_2}\int_{-ch}^0e^{-\mu_2 s}(y'(t+s)-\mu_2 y(t+s))ds= Dy. 
$$
Again, to prove the latter equality,  we have to integrate by parts. \end{proof}

\begin{definition}  Consider $(h,c) \in {\mathcal D}_\kappa$. We define the fundamental solution $\psi(t)$ of  equation 
$D_2y=\delta(t)$ where $\delta(t)$ is the Dirac $\delta$-function in the following way: 
$$\psi(t) = -\frac{\mu_1-\mu_2}{\chi'_\kappa(\mu_1)}e^{\mu_1 t}, \quad t < 0,$$
and if  $t>0$ then $\psi(t)$ coincides with solution of the functional differential equation 
$(D_2 y)(t)=0$ subject to the initial conditions 
\begin{equation}\label{ivc}
y(0) = \psi(0-)+1 = 1-  \frac{\mu_1-\mu_2}{\chi'_\kappa(\mu_1)},  \quad y(s)= \psi(s), \quad s \in [-ch,0). 
\end{equation}
In this way, $(D_2\psi)(t)=0$ for all $t \not=0$ and  $\psi(0)-\psi(0-)=1$. 
\end{definition}
\begin{lemma} \label{DDD2} The fundamental solution $\psi(t)$ is negative: $\psi(t) <0$ for all $t\in \R$
and exponentially decaying at $\pm \infty$. \end{lemma}
\begin{proof} Clearly, $\chi'_\kappa(\mu_1) >0$ and therefore 
$\psi(t) <0$ for all $t<0$.  Next, since
$$
0 = \chi_\kappa(\mu_1)-\chi_\kappa(\mu_2) = (\mu_1+\mu_2 -c)(\mu_1-\mu_2) -chg'(\kappa)e^{-\theta ch}(\mu_1-\mu_2), \quad \theta \in (\mu_2,\mu_1), 
$$
we find that 
$$
 \mu_1+\mu_2 -c=chg'(\kappa)e^{-\theta ch} < chg'(\kappa)e^{-\mu_1 ch},
$$
and therefore $\psi(0)<0$ because of 
$$
\chi'_\kappa(\mu_1) \psi(0)= \chi'_\kappa(\mu_1)  -(\mu_1-\mu_2) = \mu_1 +\mu_2 -c -chg'(\kappa)e^{-\mu_1 ch}<0. 
$$
We claim that $\psi(t) <0$ for all  $t >0$. 

\vspace{2mm}

\underline{\sl Step 1}. First, assuming that $c < c_\kappa(h)$, we  find an appropriate asymptotic representation of $\psi$ at $+\infty$. 
As a solution of linear functional differential equation,  $\psi(t)$ has at most exponential growth at $+\infty$ (see  \cite[Section 1.3]{Hale}) and therefore we 
can apply the Laplace transform method to 
$$
\psi'(t) - (c-\mu_2)\psi(t) - g'(\kappa)e^{-ch\mu_2}\int_{-ch}^0e^{-\mu_2 s}\psi(t+s)ds =0, 
$$
taking into account conditions (\ref{ivc}). Let $\Psi(z)$ denote the Laplace transform of $\psi(t)$. After some easy computations we get 
$$
0=z\Psi(z) - \psi(0)- (c-\mu_2)\Psi(z) - \frac{g'(\kappa)(e^{-zch}-e^{-\mu_2 ch})}{\mu_2-z}\Psi(z)+A(z) 
$$
$$
=\Psi(z)\frac{\chi_\kappa(z)}{z-\mu_2} - \psi(0)+A(z), 
$$
where entire function 
$A(z)$ is given by 
$$
A(z) = \frac{g'(\kappa)e^{-ch\mu_2}(\mu_1-\mu_2)}{\chi_\kappa'(\mu_1)}\int_{-ch}^0e^{-(\mu_2-z)v}dv\int_v^0e^{(\mu_1-z)s}ds.
$$ 
The following properties of $A(z)$ can be easily checked:
$$
A(\mu_1)=  \psi(0), \quad A(\mu_2)=  A(\mu_3)= \psi(0)-1.  
$$
Consequently, for each pair $(h,c)\in \R^2_+$ such that $c < c_\kappa(h)$, the function 
$$
\Psi(z) = \frac{z-\mu_2}{\chi_\kappa(z)}\left(\psi(0)-A(z)\right), 
$$
 is meromorphic  on $\C$ and analytic on the half-plane $\Re z > \mu_3$. Thus the function 
 $$
 \Psi_1(z) = \Psi(z) -\frac{\mu_3-\mu_2}{\chi'(\mu_3)(z-\mu_3)}
 $$
 is analytic on the half-plane $\Re z > \mu_3-\delta$ for sufficiently small $\delta >0$.  Observe that the fraction in the above representation 
 corresponds to the Laplace transform of the negative eigenfunction 
 $$
e_3(t)= \frac{\mu_3-\mu_2}{\chi'(\mu_3)}e^{\mu_3 t} 
 $$
 for the operators $D$ and $D_2$. 
 
 Next,  $\psi(t)-e_3(t)$ is $C^2-$smooth for $t > 2ch$ and  we find that $(D(\psi-e_3))(t) = (D_1D_2(\psi-e_3))(t) =0$ for all $t > 2ch$.  Taking into account 
 the analyticity of $\Psi_1(z)$ on the half-plane $\Re z > \mu_3-\delta$ for sufficiently small $\delta >0$, in view of \cite[Proposition 7.2]{MP}, we find that 
 $$
 \psi(t)-e_3(t) = O(e^{(\mu_3-\delta)t}), \quad t \to +\infty. 
 $$
This means that 
$$\psi(t) = \frac{\mu_3-\mu_2}{\chi'(\mu_3)}e^{\mu_3 t} + O(e^{(\mu_3-\delta)t})<0, \quad t \to +\infty,
$$
so  that $\psi(t) <0$ for all sufficiently large $t >0$.  

Suppose now that $\psi(t_0)\geq 0$ for some $t_0>0$. Consider the family of functions
$$
\psi(t,p) = \psi(t) +pe_3(t),  \quad p \geq 0. 
$$
Since $e_3(t)<0, \ t \in \R$,  there exists the smallest $p_0 \geq 0$ such that $\psi(t,p_0) \leq 0$ for all $t \geq 0$.  Then   $\psi(t_1,p_0) = \psi'(t_1,p_0)  =0$ for some leftmost $t_1>0$. 
Clearly, $D_2\psi(t,p_0) =0$ for all $t>0$. In particular, 
$$
0=D_2\psi(t_1,p_0) = \psi'(t_1,p_0) - (c-\mu_2)\psi(t_1,p_0) - g'(\kappa)e^{-ch\mu_2}\int_{-ch}^0e^{-\mu_2 s}\psi(t_1+s,p_0)ds
$$
$$ = - g'(\kappa)e^{-ch\mu_2}\int_{-ch}^0e^{-\mu_2 s}\psi(t_1+s,p_0)ds<0, 
$$
a contradiction proving that actually $\psi(t) <0$ for all $t \geq 0$. 

\vspace{2mm}

\underline{\sl Step 2}.  Now, consider the case when $\bar c=c_\kappa(\bar h)$ and take an increasing sequence of positive $c_j$ converging to $\bar c$. 
We will use the notation $\psi(t, h, c)$ to show dependence of the fundamental solution of  equation 
$D_2y=\delta(t)$ on parameters $h,c$. In view of continuous dependence of solutions of the functional differential equation $D_2y=0$ on parameters and initial data, 
we obtain  that $\psi(t, \bar h, c_j)$ converges to $\psi(t, \bar h, \bar c)$ uniformly on compact subsets of   $[0,+\infty)$. Consequently, by Step 1, we conclude that 
 $\psi(t, \bar h, \bar c) \leq 0$ for all $t \geq 0$. Suppose that $\psi(t_1,\bar h, \bar c) = 0$  for some leftmost $t_1$. Then  $t_1 > 0$ and  
$\psi'(t_1,\bar h, \bar c)  =0$.  However, arguing as at the end of Step 1, we  get immediately the contradiction $0=D_2\psi(t_1,\bar h, \bar c) <0$.  Hence, $\psi(t, \bar h, \bar c) < 0$ for all $t \geq 0$.
 \end{proof}
 Let $f:\R \to \R$ be a bounded continuous function, then the convolution 
 $$
 y(t) = \int_\R\psi(t-s)f(s)ds =-\frac{\mu_1-\mu_2}{\chi'_\kappa(\mu_1)} \int^{+\infty}_t e^{\mu_1 (t-s)}f(s)ds+\int^t_{-\infty}\psi(t-s)f(s)ds,
 $$
 is bounded and continuously differentiable function satisfying  the functional differential equation $D_2y=f$.  We can use this 
 fact to solve the second-order equation $Dy=f$. Indeed, the equation $D_1(D_2y)=f$ has a unique bounded solution 
 $$
 (D_2y)(t) = \int_{-\infty}^te^{\mu_2(t-s)}f(s)ds= (\theta*f)(t), 
 $$ 
where
$$
\theta(t) = e^{\mu_2 t}, \quad t \geq 0, \qquad  \theta(t) =0, \quad t <0,  
$$
is {\it the fundamental solution} of $D_1y=\delta(t)$. Consequently, the equation $D_1(D_2y)=f$ has a unique bounded solution 
$$
y = \psi*(\theta*f)= (\psi*\theta)*f. 
$$
The function $N=\psi*\theta$ is called the {\it fundamental solution} of the equation $Dy=\delta(t)$,
clearly, $N(t)<0$ for all $t \in \R$, this property was earlier established in  \cite{TpT} by using an alternative  (and more technical) approach.

\section{Proof of Theorem \ref{main1}}\label{Sp1}

By \cite[Theorem 8]{TpT}, for $h \in [0,h_*]$  each traveling wave $u(t,x)=\phi(x+ct)$  has  strictly increasing profile $\phi(t)$, moreover, $\phi'(t) >0$ for all $t \in \R$ (see also   \cite[Lemma 6]{TTT}).  The same theorem in \cite{TpT} assures that $\phi(t)$ is unique up to  translation. Next, if $h > h_*$ then there exists $\tilde c> c^*(h)$ such that $(h, \tilde c) \not \in   {\mathcal D}_\kappa$. Due to Theorem 1.7 in \cite{LMS},  traveling wave propagating with the speed $\tilde c$ is not monotone. This establishes the optimal nature of the bound $h_*$. 
In this way, we  have only to prove that for each fixed  $h \in [0,h_*]$, the set of all possible wave speeds is a
connected interval of the form $[c_*(h), +\infty)$. 
The next assertion provides one of the key arguments for the proof of   Theorem \ref{main1}: 
\begin{lemma} \label{LAP} Suppose that $g$ satisfies  (\ref{gcos})  with {\rm \bf(U)} and $(h,c) \in {\mathcal D}_\kappa$  (so that $\phi'(t) > 0$)
be such that 
\begin{equation}\label{hka}
1+hg'(\kappa)e^{-\mu_2(c) ch} >0. 
\end{equation}
Then 
$$
\phi'(t)+  hg'(\kappa)\phi'(t-\bar ch) >0, \quad t \in \R
$$
for  every wavefront $\phi(x+ct)$ of equation (\ref{MGD})  which propagates with the speed  $c$ and each $\bar c $  for which $1+hg'(\kappa)e^{-\mu_2(c) \bar ch} >0$. 
\end{lemma}
\begin{proof} Set 
$$
{\mathcal D}_*=\{(h,c) \in \R_+^2: 1+hg'(\kappa)e^{-\mu_2(c) ch} >0\}. 
$$
Then the boundary of ${\mathcal D}_*$ consists of the union of the  half-line $(0,c), c \geq 0$, with  the interval $(0,h), 0 \leq h <-1/g'(\kappa)$, and with the graph $\Gamma_1$
defined by the system of equations 
\begin{equation}\label{sy1}
1+hg'(\kappa)e^{-\mu ch} =0, \ \mu^2-c\mu-1+g'(\kappa)e^{-c\mu h} =0. 
\end{equation}
On the other hand, the graph $\Gamma_2$ of $c=c_\kappa(h)$ is defined by the system 
$$
2\mu -c -chg'(\kappa)e^{-\mu ch} =0, \ \mu^2-c\mu-1+g'(\kappa)e^{-c\mu h} =0. 
$$ 
If $\Gamma_1$ and $\Gamma_2$ intersect at some point $(h',c')$ then necessarily $\mu =0$, a contradiction. 
Thus we conclude that 
$\Gamma_1$ belongs to  the interior of the set ${\mathcal D}_\kappa$.  In fact, solving (\ref{sy1}), we find that 
$$
\mu= \frac{c-\sqrt{c^2+4+4/h}}{2}= \frac{1}{ch} \ln(h|g'(\kappa|), 
$$
from which we obtain the equation $c=c(h)$ for the curved part of the boundary of ${\mathcal D}_*$:
\begin{equation}\label{c(h)}
c(h)= \frac{-\ln(h|g'(\kappa)|)}{\sqrt{h(1+h+\ln(h|g'(\kappa)|))}}, \quad h_* < h \leq h^*:=1/|g'(\kappa)|. 
\end{equation}
Clearly, $c=c(h)$ is strictly decreasing function with $c(h_*^+)=+\infty, \ c(h^*)=0$.  Therefore condition (\ref{hka}) is automatically satisfied for all $h \in [0,h_*]$. 

Now, if $u(t,x) = \phi(x +ct)$ is a monotone wavefront to (\ref{MGD}), then $\phi$ satisfies 
$$
\phi''(t) - c\phi'(t) - \phi(t) + g(\phi(t-ch)) =0. 
$$
Set $z(t)= \phi'(t)>0$. By differentiating the latter equation, we find that 
$$
z''(t)-cz'(t) - z(t) +g'(\phi(t-ch))z(t-ch)=0, 
$$
or, equivalently, 
 \begin{equation*}\label{e41}
(D_2D_1)z(t)= z''(t)-cz'(t) - z(t) + g'(\kappa)z(t-ch)=b(t),
\end{equation*}
where 
$$
b(t)= a(t)z(t-ch) \leq 0, \quad  a(t)= g'(\kappa)- g'(\phi(t-ch)) \leq 0,\quad t \in \R. 
$$
Consequently, 
$$
(D_1z)(t) =z'(t) -\mu_2 z(t) = (\psi*b)(t) \geq 0, 
$$
so that $(z(t)e^{-\mu_2 t})' \geq 0$ and 
$$
z(t-\bar ch)e^{-\mu_2(t-\bar ch)} = \phi'(t-\bar ch)e^{-\mu_2(t-\bar ch)} \leq  \phi'(t)e^{-\mu_2 t} = z(t)e^{-\mu_2 t}, \quad \ t\in \R. 
$$
Hence, $ \phi'(t-\bar ch)e^{\mu_2 \bar ch} \leq  \phi'(t)$ and 
$$
\phi'(t)+  hg'(\kappa)\phi'(t-\bar ch) \geq  \phi'(t-\bar ch)e^{\mu_2 \bar ch} +  hg'(\kappa)\phi'(t-\bar ch) $$
$$= \phi'(t-\bar ch)(e^{\mu_2 \bar ch} +  hg'(\kappa))>  0, \quad t \in \R. 
$$
This completes the proof of Lemma \ref{LAP}. 
\end{proof}
\begin{corollary}\label{Cor1} Let wavefront $u= \phi(x+c_0t)$  be such that  $c_0, h$ satisfy (\ref{hka}). Then 
 $$
\phi'(t)+  hg'(\kappa)\phi'(t-\hat ch) >0, \quad t \in \R, 
$$
whenever $\hat c \in [c_0,c_0+\nu]$ and $\nu >0$ is sufficiently small number. 

\end{corollary}
Now, fix some $h \in  (0,h_*]$ and consider 
$$
{\mathcal{C}}(h) := \{ c \geq 0: \ {\rm equation} \ (\ref{MGD}) \ {\rm has\ a\ wavefront \ propagating \ at \ the \ velocity\ } c \}.
$$
It is known from \cite{tt}, that  ${\mathcal{C}}(h) $ contains the subinterval $[c^*(h), +\infty)$ while 
$$
c_*: =\inf {\mathcal{C}}(h) \geq c_\# >0. 
$$
It is easy to see that ${\mathcal{C}}(h)$ is closed so that $c_* \in {\mathcal{C}}(h)$. 
Assume  that 
$c_0 \in {\mathcal{C}}(h)\cap [c_\#(h), c^*(h))$ and let $u(t,x) = \phi(x +c_0t)$ be  a wavefront solution, 
$$
\phi''(t) - c_0\phi'(t) - \phi(t) + g(\phi(t-c_0h)) =0. 
$$
Then take $\nu$ as in Corollary \ref{Cor1} and let 
$c' \in  [c_\#(h), c^*(h))$, $c' - c_0 \in (0,\nu)$ be small enough to satisfy  $(1+\gamma) \lambda_2(c') > \lambda_2(c_0)$.
Note that $\lambda_2(c)$ is a decreasing function of $c$.  
To simplify the notation, we will write  $\lambda_2' := \lambda_2(c'), \ \lambda_2:=\lambda_2(c_0)$. For the reader's convenience, the proof of of Theorem \ref{main1} is divided into several steps.  

\vspace{2mm}

\noindent \underline{Step I (Properties of an auxiliary function $\phi_\sigma$).} Set $\phi_\sigma(t): = \sigma \phi(t)$, where  $\sigma >1$ is close to $1$. We have 
$$
E(t,\sigma):= \phi_\sigma''(t) - c'\phi'_\sigma(t) - \phi_\sigma(t) + g(\phi_\sigma(t-c'h)) =
$$
$$
\phi_\sigma''(t) - c_0\phi'_\sigma(t) - \phi_\sigma(t) + \sigma g(\phi(t-c_0h)) + [(c_0-c')\phi'_\sigma(t)+ g(\phi_\sigma(t-c'h)) - \sigma g(\phi(t-c_0h))] = 
$$
$$
 (c_0-c')\phi'_\sigma(t)+ g(\phi_\sigma(t-c'h)) - \sigma g(\phi(t-c_0h)).
$$
By our assumptions, $|g(x) - g'(0)x|\leq Ax^{1+\gamma}, \ x \in[0,\delta]$. Take some  $t_*$ such that $\phi(t_*-c_0h)< \delta/2$. Then 
for all $1< \sigma  \leq 2$ and $t \leq t_*$, 
$$
g(\phi_\sigma(t-c'h)) - \sigma g(\phi(t-c_0h))\leq $$
$$
g(\phi_\sigma(t-c_0h)) - \sigma g(\phi(t-c_0h))\leq 6A(\phi(t-c_0h))^{1+\gamma}. 
$$
On the other hand, from Proposition \ref{pr1p} we know that 
$$
 (c_0-c')\phi'_\sigma(t)=  (c_0-c')\zeta \sigma\phi(t) (1+ o(1)) = (c_0-c')\zeta \sigma e^{\zeta c_0 h}\phi(t-c_0h) (1+ o(1))
$$
for $\zeta \in \{\lambda_1(c_0), \lambda_2(c_0)\}$ and $t \to -\infty$. 
As a consequence, there exists $T_1\leq t_*$ (which does not depend on $\sigma$) such that, for all $\sigma \in (1,2]$,  
$$
E(t,\sigma) < 0,\  t \leq T_1. 
$$
Due to assumption {\rm \bf(U)}, the function 
$
G(u) : = g(u)/u 
$
has negative derivative on some interval  $\mathcal O= (\theta',+\infty) \supset [\theta,\kappa]$.  Thus  
$$
G(u)-G(v) = G'(w)(u-v) < 0,\quad  u > w> v \geq \theta'. 
$$
Observe  that $\theta'$  does not depend on $\sigma$. Since from the very beginning we can fix $c'$ sufficiently 
close to $c_0$ to have $\phi(T_2-c'h), \phi(T_2-c_0h)  \in (\theta',\theta)$,  for some $T_2$ (which depends only on $\phi,  c_0$), we obtain that  
for $t\geq T_2$ and every $\sigma >1$, it holds 
$$
 E(t,\sigma) = (c_0-c')\phi'_\sigma(t)+ g(\phi_\sigma(t-c'h)) - \sigma g(\phi(t-c_0h)) = 
$$
$$
 (c_0-c')\phi'_\sigma(t)+ \sigma\phi(t-c'h)(G(\phi_\sigma(t-c'h)) - G(\phi(t-c'h)))+  \sigma( g(\phi(t-c'h)) -  g(\phi(t-c_0h))) \leq 
$$
$$
 (c_0-c')\left(\phi'_\sigma(t)+  \sigma hg'(\kappa)\phi'(t-\hat ch) \right) =
$$
$$
 (c_0-c')\sigma\left(\phi'(t)+  hg'(\kappa)\phi'(t-\hat ch) \right) <0, \quad \mbox{for some} \ \hat c \in (c_0,c'). 
$$
Finally, since $\phi_\sigma(t-c'h), \phi_\sigma(t-c_0h)  \in (0,\theta)$, for all small $\sigma$ and $t \leq T_2$, we conclude that
$$
  E(t,\sigma)= (c_0-c')\phi'_\sigma(t)+ g(\phi_\sigma(t-c'h)) - \sigma g(\phi(t-c_0h))  <
 $$
$$
 (c_0-c')\phi'_\sigma(t)+ g(\phi_\sigma(t-c_0h)) - \sigma g(\phi(t-c_0h)) <0
$$
uniformly on $[T_1,T_2]$ for all small $\sigma -1>0$. 

All the above shows that $E(t,\sigma) <0$ for all $t \in \R$ and each $\sigma >1$ sufficiently close to $1$. 

\vspace{2mm}

\noindent  \underline{Step II (Construction of an upper solution).}  
By  Step I, we can choose $c', \sigma> 1$ in such a way that  $E(t,\sigma) <0$, $t \in \R$. 
For $a:=b^2,\  b \in (0,1]$, set $\phi_{b}(t): = \phi_\sigma(t) + ae^{\lambda'_2 t}+ be^{\lambda_2 t}$, where $\lambda'_2 =\lambda_2(c'),$ $ \ \lambda_2 =\lambda_2(c_0)$. Let $T_3= T_3(b)$ be that unique point where 
$\phi_b(T_3(b))=\kappa$, then $T_3(b) \leq T_3(0)$ for all $b\geq 0$. It is clear that $\phi'_b(T_3) >0$ and that $T_3(b) \to T_3(0)$ as $b \to 0$.  
 Next, we find  that 
$$
E_+(t,b):= \phi_b''(t) - c'\phi'_b(t) - \phi_b(t) + g(\phi_b(t-c'h)) =   E(t,\sigma)+ b\chi(\lambda_2,c')e^{\lambda_2 t} +
$$
$$
g( \phi_\sigma(t-c'h) + ae^{\lambda'_2 (t-c'h)}+be^{\lambda_2(t-c'h)})-   g(\phi_\sigma(t-c'h)) - g'(0)(ae^{\lambda'_2 (t-c'h)}+be^{\lambda_2(t-c'h)})
$$
$$
 \leq E(t,\sigma) +b\chi(\lambda_2,c')e^{\lambda_2 t}+$$
$$
A(ae^{\lambda'_2 (t-c'h)}+be^{\lambda_2(t-c'h)})( \phi_\sigma(t-c'h) + ae^{\lambda'_2 (t-c'h)}+
be^{\lambda_2(t-c'h)})^\gamma \leq  E(t,\sigma) +
$$
$$
be^{\lambda_2 t}\left(\chi(\lambda_2,c')+3Ae^{-\lambda_2 c'h}(be^{(\lambda'_2-\lambda_2) (t-c'h)}+1)
(\phi^\gamma_\sigma(t-c'h)+2b^\gamma e^{\lambda'_2\gamma (t-c'h)} )\right)
$$
$$
\leq E(t,\sigma) +be^{\lambda_2 t}\left(\chi(\lambda_2,c')+C_1 
\phi^\gamma_\sigma(t-c'h)+C_2b^\gamma e^{(\lambda'_2(1+\gamma)-\lambda_2) (t-c'h)} \right)$$
$$
+be^{\lambda_2 t}C_3b \phi^\gamma_\sigma(t-c'h)e^{(\lambda'_2-\lambda_2)t }\leq 
E(t,\sigma) + be^{\lambda_2 t}\left(\chi(\lambda_2,c')+C_4 e^{\nu_1 t}\right), \quad t \leq T_4,
$$
for some positive $\nu_1, \ C_j$ and negative $T_4$ (which does not depend on $b$).   Since $\chi(\lambda_2,c')<0$, we may choose  $T_4$ is  such a way that 
$E_+(t,b) <0$ for all $t \leq T_4,$ $ b \in (0,1]$.  On the other hand, we know that, uniformly on each compact interval,  $E_+(t,b) \to E(t,\sigma),  b \to 0+$.  Therefore $E_+(t,b) <0$ for all $t \leq T_3(0+)+1$ for all sufficiently small $b$. 

\noindent Consider now
$C^\infty$-smooth non-increasing function $\psi(t)$ such that $\psi(t) = 1$ for all $t \leq T_3(0)+1$ and 
$\psi(t) = 0$ for all $t \geq T_3(0)+1+c'h$. 
We define an upper solution $\phi_+$ by 
$$
\phi_+(t): =   \phi_\sigma(t) + (ae^{\lambda'_2 t}+ be^{\lambda_2 t})\psi(t). 
$$
Observe that, for all small $b$, the function $\phi_+(t)$ is increasing and 
$$
 \phi_+''(t) - c'\phi'_+(t) - \phi_+(t) + g(\phi_+(t-c'h)) =  
 \left\{\begin{array}{cc} 
 E(t,\sigma) <0,&  \mbox{for all} \ t \geq T_3(0)+1+c'h; \\    
E_+(t,b)<0, & t \leq T_3(0)+1.\end{array}\right. 
$$
Since uniformly on $[T_3(0)+1,T_3(0)+1+c'h]$, 
$$
\lim_{b\to 0+} ( \phi_+''(t) - c'\phi'_+(t) - \phi_+(t) + g(\phi_+(t-c'h)) ) =  E(t,\sigma) <0, 
$$
we conclude that, for all small $b>0$,  
$$
 \phi_+''(t) - c'\phi'_+(t) - \phi_+(t) + g(\phi_+(t-c'h)) <0, \quad t \in \R. 
$$

\noindent \underline{Step III (Construction of a lower solution).} 
Consider the  following concave monotone  linear rational  function 
$$p(x):= \frac{g'(0)x}{1 +Bx} \leq g'(0)x, \ x \geq 0,  \ B: =  \frac{g'(0)-1}{\theta}, \ p(0)=0, \ p(\theta)= \theta,$$
and set $g_-(x) : = \min \{g(x),p(x)\}$. It is clear that $g_-$ is continuous and increasing on $[0,\theta]$ and that 
$$g_-'(0)= g'(0), \ g_-(0)=0, \ g_-(\theta)= \theta, \quad g_-(x) \leq g'(0)x, \ x \geq 0.$$
Moreover, in some right neighborhood of $0$,   
\begin{equation*}
\left|g_-(u)/u- g'(0)\right| \leq A'u^\gamma, \quad u\in  (0,\delta'],
\end{equation*}
for some $A' >0,\  \delta' >0$.    
As we have mentioned in the introduction, this implies the existence of a monotone positive function $\phi_-, \ \phi_-(-\infty)=0,$ $\phi_-(+\infty)= \theta,$ satisfying the equation
$$
\phi_-''(t) - c'\phi_-'(t) - \phi_-(t) + g_-(\phi_-(t-c'h)) =0. 
$$
Due to the  property $g_-(x) \leq g'(0)x, \ x \geq 0$, we also know that 
$$
(\phi_-, \phi'_-)(t+t_0,c)= e^{\lambda'_2 t}(1, \lambda'_2) + O(e^{(\lambda'_2+ \sigma) t}), \ t \to -\infty. 
$$
Finally,  since $g_-(x) \leq g(x)$ we obtain that
$$
\phi_-''(t) - c'\phi_-'(t) - \phi_-(t) + g(\phi_-(t-c'h)) \geq 0. 
$$
\underline{Step IV (Iterations).} 
Comparing asymptotic representations of monotone functions  $\phi_-(t)$ and $\phi_+(t)$  at $+\infty$ and $-\infty$, we   find easily that 
$$
\phi_-(t+s_1) \leq  \phi_+(t), \quad t \in \R, 
$$ 
for some appropriate $s_1$. 
Simplifying, we will suppose that $s_1=0$.  

Using the fundamental solution $N(t), \ \int_\R N(s)ds= 1/(g'(\kappa)-1)$,  defined in Section \ref{fund}, we can rewrite the profile equation (\ref{prfe}) with $c=c'$ in the following equivalent forms 
\begin{equation*}
(D'\phi)(t) = g'(\kappa)\phi(t-c'h) -g(\phi(t-c'h)), \quad \phi(-\infty)=0, \ \phi(+\infty) = \kappa, 
\end{equation*}
$$
\hspace{-7mm}(\mbox{here we use the notation} \quad (D'y)(t):=y''(t) - c'y'(t) - y(t) + g'(\kappa)y(t-c'h)), 
$$
and $\phi(t) = (\mathcal{N}\phi)(t),  \  \phi(-\infty)=0, \ \phi(+\infty) = \kappa$, where 
$$
 (\mathcal{N}\phi)(t): =  \int_\R N(t-s) \left(g'(\kappa)\phi(s-c'h)-  g(\phi(s-c'h))\right)ds. 
$$
Since $N(t) < 0, \  t \in \R,$ and, by our assumption, $\min_{u\in[0,\kappa]}g'(u)= g'(\kappa) < 0,$ the integral operator $\mathcal{N}$ 
is increasing on $C(\R, [0,\kappa])$, i.e. 
$$
 0 \leq (\mathcal{N}\phi)(t)\leq  (\mathcal{N}\psi)(t) \leq \kappa, \quad \mbox{whenever} \ \phi(t) \leq \psi(t), \ t \in \R, \quad \phi, \psi \in C(\R, [0,\kappa]). 
$$
In addition, the properties of  functions $\phi_-(t)$ and $\phi_+(t)$ guarantee  that 
$$
\phi_-(t) \leq (\mathcal{N}\phi_-)(t) \leq (\mathcal{N}^2\phi_-)(t)\leq \dots \leq (\mathcal{N}^2\phi_+)(t) \leq  (\mathcal{N}\phi_+)(t) \leq    \phi_+(t), \quad t \in \R. 
$$
By standard arguments (e.g. see  \cite{TpT} for details),  the latter  implies  the existence of a monotone continuous function  $\phi(t)=\lim_{k\to +\infty} (\mathcal{N}^k\phi_+)(t)$ such that 
$$
(\mathcal{N}\phi)(t) = \phi(t),  \quad \phi_-(t) \leq  \phi(t) \leq  \phi_+(t), \quad t \in \R. 
$$
This amounts to the existence of a wavefront propagating at velocity $c'$. Moreover, the latter estimations show that, for some $s_0$ and positive $\delta$, 
\begin{equation}\label{ff}
\phi(t+s_0)= e^{\lambda'_2 t} + O(e^{(\lambda'_2+ \delta) t}), \ t \to -\infty.
\end{equation}
Finally, to prove that ${\mathcal{C}}(h)$ coincides with the interval $[c_*, \infty)$, let us consider the open set 
$O= [c_*,\infty)\setminus {\mathcal{C}}(h)$.  If $O \not=\emptyset$, we  take one connected component 
of $O$, say $(c_0,c_1)$.  Since $c_0 \in {\mathcal{C}}(h)$, there is some  $c'\in (c_0,c_1)$ such that 
$c' \in {\mathcal{C}}(h)$, in contradiction to the definition of $O$. Therefore ${\mathcal{C}}(h)=[c_*, \infty)$. 
\begin{remark} \label{reS} (I) It is easy to see that the result established in this section is slightly stronger that the assertion of  Theorem \ref{main1}. Indeed, we have proved that 
even for $h> h_*$ sufficiently close to $h_*$ there is a positive 
number $c_*=c_*(h)$ such that  equation (\ref{MGD}) (a)  possesses  a unique monotone wavefront $u(t,x)= \phi_c( x+ct)$ for every $c \in [c_*,c(h)]$, where $c(h)$ is given by (\ref{c(h)}); (b) does not have  any semi-wavefront propagating at the velocity $c < c_*$. It is instructive to note that the main conclusion of Theorem 1.3 in \cite{ai} is  of the same kind. 

(II)  Take some monotone wavefront $\phi(x+c_0t)$ and consider the expression 
$$
F(t,c):= \phi''(t) - c\phi'(t) - \phi(t) + g(\phi(t-ch)). 
$$
Then $F(t,c_0)\equiv 0$, $$F_c(t,c_0) = -(\phi'(t)+hg'(\phi(t-c_0h))\phi'(t-c_0h)),$$ so that if we want the inequality $
E(t,\sigma) <0$ to be satisfied for all small $c'-c_0 >0,$ $\sigma-1>0$, we need to assure the following property of a wavefront: $$
\phi'(t)+hg'(\phi(t-c_0h))\phi'(t-c_0h) >0, \quad t \in \R.
$$
Next, it is well known (e.g. see \cite[Lemma 4.3]{LMS}) that the monotonicity of $\phi(t)$ and the assumption $g'(x) \geq g'(\kappa)$, $x \in [0, \kappa]$, imply that 
$$
\lim_{t \to +\infty}g'(\phi(t-c_0h)) = g'(\kappa), \quad  \lim_{t \to +\infty}[\phi'(t-c_0h)/\phi'(t)] = e^{-\mu_2 c_0h}.  
$$
All this shows that condition (\ref{hka}) is nearly optimal for the construction of an upper solution for the perturbed profile equation (with $c' >c_0$ close to $c_0$) from 
$\phi(t)$. 
\end{remark}

\section{Proofs of Proposition \ref{pr1p} and Theorem \ref{pr1}}\label{poofs} 

\begin{proof}[Proof of Proposition \ref{pr1p}] Due to \cite[Lemma 6]{TTT}, we have that  $\phi'(s) >0$ for all $s$ from some infinite interval $(-\infty,\sigma).$  The exponential decay  of $\phi(t)$ at $ -\infty$ is assured by  \cite[Lemma 3 (ii)]{AGT}.   Therefore there is $\delta >0$ such that 
$$
g(\phi(t-ch))= \left[g'(0) + r(t)\right]\phi(t-ch),\ \text{where} \ r(t):= \frac{g(\phi(t-ch))}{\phi(t-ch)} -g'(0) = o(e^{\delta t}).
$$
On the other hand, it is easy to see that the convergence  $\phi(t) \to 0, t \to -\infty,$  is not super-exponential, cf.  \cite[Theorem 5.4 and Remark 5.5]{tt}.  Now we can proceed as in  \cite[Remark 5.5]{tt} (where \cite[Proposition 7.2]{MP} should be used) to obtain  asymptotic formulas (\ref{afe}), (\ref{afec}).  
\end{proof}
\begin{proof}[Proof of Theorem \ref{pr1}] It follows from Theorem \ref{main1} and Remark \ref{reS} (I) that we can associate 
a unique monotone traveling front $u=\phi_c(x+ct)$ satisfying the normalization condition $\phi_c(0)=\theta$
 to each pair $(h,c) \in {\mathcal D}_*$  such that  $c\geq c_*(h)$. The uniqueness of solutions $\phi_c(t)$ implies that $\phi_c(t)$ is a continuous function of $c,t$. Therefore for every  $c'>  c_*(h)$, $(h,c') \in {\mathcal D}_*$,  there exists some interval $[c_0,c']$ with $c_0> c_*(h)$ such that 
 $$
1+hg'(\kappa)e^{-\mu_2(c_0) \bar ch} >0, \quad \mbox{for all} \  \bar c \in [c_0, c'], $$
 as well as 
 $\phi_{c_0}(T_2-c'h), \phi_{c_0}(T_2-c_0h)  \in (\theta',\theta)$,  for some $T_2$.  Arguing now as in Section \ref{Sp1}, we find that 
 $\phi_{c'}(t) \geq \phi_-(t+t_*), \ t \in \R$, for some $t_*$. This yields the representation (\ref{ff}) and, consequently, the required formula (\ref{afep}). 
 
 Next, suppose that $(h,c_*) \in {\mathcal D}_*$ and $c_*=c_*(h)> c_\#(h)$. Proposition \ref{pr1p} shows that if (\ref{afecp}) does not hold for these parameters then the critical profile 
 $\phi_{c_*}(t)$ must satisfy the asymptotic relation (\ref{afep}). Assuming this relation to hold, we find from Theorem 1.7 in \cite{LMS}  that equation (\ref{MGD}) has a monotone wavefront $\phi_c(x+ct)$ for all $c$ from some open interval containing $c_*$ (note that the first part of  our proof together with the latter assumption imply  that the set $\mathcal{D}_{\frak N}$ defined in \cite{LMS} contains the interval 
 $\{h\} \times [c_*(h), c(h))$\ ). However, this contradicts the minimality property of $c_*(h)$. 
 \end{proof}

\section{A toy model: delay turns pushed waves into pulled waves}\label{Stoy}

In this section, aiming at understanding the structure of the set of all 
admissible speeds $\mathcal C(h)$ defined in the introduction and, partially, the dynamics of model (\ref{MGD}), we consider the following piece-wise linear  equation
\begin{eqnarray}\label{MGDt}\hspace{8mm}
u_t(t,x) - u_{xx}(t,x) + u(t,x)= \left\{\begin{array}{cc} 
ku(t-h,x),& u(t-h,x) \in [0,1), \\    
4 - u(t-h,x), & u(t-h,x) \geq 1,\end{array}\right. 
\end{eqnarray}
where  the slope $k$ of the birth function at zero satisfies $k \in (1, 3)$ and $\kappa= 2$ is the positive equilibrium (see Fig. \ref{GeomVyznam}, note that condition (\ref{subt})
holds if we choose $k \geq 3$). 
 This kind of equations,  sometimes nicknamed `toy models', is frequently used in the theory of traveling waves, cf. \cite{KH,NPTT,TVN}.
The advantage of (\ref{MGDt}) is that such a remarkable  and difficult to detect solution  of  the delayed equation such as  its positive traveling wave, can be  found explicitly by using the Laplace transform.  Now, since (\ref{MGDt}) has discontinuous right hand-side,  we define the positive  profile $\phi(t), \ \phi(-\infty)=0,$  $\phi(+\infty) = 2$,  of its wavefront $u(t,x)=\phi(x+ct)$  as $C^1-$smooth and piece-wise analytical solution of the delayed differential equation 
\begin{equation}
\phi''(t)-c\phi'(t) -\phi(t) + g(\phi(t-ch))=0, \ \mbox{where} \ g(u)=  \left\{\begin{array}{cc} 
ku,& u \in [0,1), \\    
4 - u, & u \geq 1. \end{array}\right. \label{MGDp}
\end{equation}
\begin{lemma} \label{le3} If $\phi(t)$ is a wavefront profile  for (\ref{MGDt}), then $\phi(t) < 3$ for all $t \in \R$. 
\end{lemma}
\begin{proof} By standard arguments, we obtain the following integral representation 
$$
\phi(t)= \int_\R K_1(t-s)g(\phi(s-ch))ds, \quad \mbox{where} \ K_1(s) > 0, s \in \R, \int_\R K_1(s)ds=1.  
$$
Thus the lemma follows from the fact that $g(\phi(s-ch))\leq 3$ and $g(\phi(-\infty))=0$. 
\end{proof}
We  first consider $h=0$.  It is straightforward to see that the roots of the characteristic equation  at the positive equilibrium, $\chi_\kappa(z)= z^2 -c z - 1-e^{-zch} = 0$, $h=0$,  are given by the formula 
$$
\mu_{1,2}(c): = \frac{1}{2}\left(c \pm \sqrt{c^2+8}\right)
$$
and satisfy $\mu_2 < 0 < \mu_1$,
while the roots of the characteristic equation at the zero steady state,    $\chi(z,c)= z^2 -c z -1+ke^{-zch} = 0$, $h=0$, $c\geq c_\# = 2\sqrt{k-1}$,   are 
$$
\lambda_{1,2}(c): = \frac{1}{2}\left(c \pm \sqrt{c^2-4(k-1)}\right), 
$$
where  $0<\lambda_2 \leq  \lambda_1.$
Therefore,  if $c\geq c_\#(k) = 2\sqrt{k-1}$, then the point $(0,0)$ on the phase plane $(\phi,\phi')$ is an unstable node and the point $(2,0)$ is a saddle point. 
Clearly,  each wave profile $\phi(t)$ corresponds to a unique heteroclinic connection $(\phi,\phi')$ between these equilibria of (\ref{MGDp}) on the phase plane diagram. 
Since the stable manifold of the saddle point is given by the equation $\phi'=\mu_2(\phi-2)$, we obtain easily its value $-\mu_2$ at $\phi=1$, i.e.
$$
\phi'|_{\phi=1}=\frac{1}{2}(\sqrt{c^2+8}-c).
$$ 
Next, in the non-delayed case,  the graph $(\phi,\phi')$  of the pushed wave is given by $\phi'= \lambda_1\phi$ for $\phi \leq 1$, see \cite[Section 2.6]{ES}. 
Since, in addition,  the profile $\phi$ is $C^1$-continuous, we obtain the following compatibility condition at $\phi=1$: 
$$
\lambda_1=\frac{1}{2}\left(c + \sqrt{c^2-4(k-1)}\right) = \frac{1}{2}\left(\sqrt{c^2+8}-c\right).
$$
It is easy to see that this equation can be solved only when $k\leq 5/3$ yielding 
the following relation between the speed of the pushed wave $c_*(k)$ and the slope $k$: 
$$
c_*(k) = \frac{1+k}{\sqrt{2(3-k)}}, \quad k \in \left(1,\frac 53\right], 
$$
which is applicable  since  $c_*(k) >  c_\#(k)$ for $k \in (1,5/3)$. 

Hence, in general, the minimal  speed $c_*$ of propagation in the non-delayed model (\ref{MGDt}) is given by 
\begin{eqnarray}\label{111}
c_* =  \left\{\begin{array}{ccc} 
\frac{1+k}{\sqrt{2(3-k)}},& k \in (1,5/3], & \mbox{(pushed critical wave),} \\    
2\sqrt{k-1}, &  k > 5/3, & \mbox{(pulled critical wave)}.\end{array}\right. 
\end{eqnarray}
Now, if  $k \in (1,5/3)$, then the continuity argument suggests  that  equation (\ref{MGDt})  has pushed minimal wavefront for all small  $h>0$. Characteristic property (\ref{afecp})  of this wave suggests how  its explicit determination can be obtained. 
Indeed,  let $\phi(t)$ be the profile of the minimal front propagating with the speed $c=c_*>c_\#$. 
Obviously, 
there exists  the rightmost $t_0$ such that $\phi(t)\in (0,1]$ for all  $t \leq t_0-ch$, and $\phi(t_0-ch)=1$. Set for simplicity $t_0=0$. 
Then for all  $t \leq 0$, the wave profile $\phi$ is a positive solution of the linear equation
$$
\phi''(t)-c\phi'(t)-\phi(t) + k\phi(t-ch)=0.
$$
For $c >c_\#$,  the characteristic equation $\chi(z,c)=0$  has two positive real roots $0<\lambda_2 < \lambda_1$ which dominate each other (complex) root  in the sense that 
$\Re \lambda_j < \lambda_2$, e.g. see   \cite[Lemma 2.3]{tt}. Then the positivity of $\phi$ and its pushed character imply that 
\begin{equation*}\label{solutionMGDtp}
\phi(t)=e^{\lambda_1(t+ch)}, \quad t\leq 0.
\end{equation*}
Let us suppose now that $\phi(t) \geq 1$ for all $t \geq -ch$. This assumption is automatically satisfied if $\phi(t)$ is monotone on $\R$ and we will also prove  later in this section   that each wavefront (not necessarily pushed) to  (\ref{MGDt}) normalized by $\phi(-ch) = 1$ has to satisfy $\phi(t) > 1$ for all $t >-ch$. Then for  $t>0$ the profile  $\phi$ verifies    
\begin{equation}\label{MGDtp}
\phi''(t)-c\phi'(t)-\phi(t)+4 - \phi(t-ch)=0.
\end{equation}
The change of variables $\psi=\phi-2$ transforms the latter equation into
\begin{equation}\label{MGDtp2}
\psi''(t)-c\psi'(t)-\psi(t) - \psi(t-ch)=0.
\end{equation}
The $C^1$-continuity of $\phi$ also implies that 
$$
\psi(t) = e^{\lambda_1(t+ch)}-2,\quad t \in [-ch,0],
$$
$$
\psi(0)= e^{\lambda_1ch}-2, \quad \psi'(0) = \lambda_1 e^{\lambda_1ch}.
$$
Applying the Laplace transform $(L\psi)(z)= \int\limits_0^{\infty}e^{-zt}\psi(t)dt$ to  (\ref{MGDtp2}), we get
\begin{equation}\label{laplaceMGDtp2}
\chi_\kappa(z)(L\psi)(z)=\psi'(0)+z\psi(0)-c\psi(0)+e^{-zch}\int\limits_{-ch}^{0}e^{-zt}\psi(t)dt.
\end{equation}
By using the Rouch\'e theorem, it is easy to find  that $\chi_\kappa(z)$  has a unique positive zero $\mu_1$ while other characteristic values have negative real parts. Therefore $\lim \psi(t)=0$, $t \to +\infty$, if and only if $(L\psi)(\mu_1)=0.$ The last equation has the form
$$
\lambda_1 e^{\lambda_1ch} + (\mu_1-c)(e^{\lambda_1ch}-2)+ e^{-\mu_1 ch}\int\limits_{-ch}^{0}e^{-\mu_1 t}\psi(t)dt = 0, 
$$
which can be transformed (by using the relations  $\chi(\lambda_1,c) = \chi_\kappa(\mu_1)=0$)  into
\begin{equation}\label{theta}
\frac{\lambda_1(c)}{\mu_1(c)}=\frac{3-k}{4}.
\end{equation}
A simple calculation shows that formula (\ref{theta}) agrees with \eqref{111} with $h=0$ while our previous discussion suggests  that 
(\ref{theta}) must have at least one solution $c_*(h)$ close to $c_*(0)$ for all small $h>0$.  For further analysis of  (\ref{theta}) 
note that, being simple zeros of the characteristic quasi-polynomials, 
$\lambda_1(c)$ and $\mu_1(c)$ are analytical functions of $c$. In addition,  
$T(c):= \lambda_1(c)/{\mu_1(c)}$ is an increasing function of $c$. Indeed, 
$$
T(c):= \frac{\lambda_1(c)}{\mu_1(c)}=\frac{\tilde \lambda_1(c) - \epsilon_1(c)}{\tilde \lambda_1(c)+\epsilon_2(c)} = 
\frac{1-{\frac{2\epsilon_1(c)}{\sqrt{c^2+4}+c}}}{1+{\frac{2\epsilon_2(c)}{\sqrt{c^2+4}+c}}},
$$
where  $\tilde \lambda_1(c)=(\sqrt{c^2+4}+c)/2$ satisfies $z^2 -cz - 1=0$  and $\epsilon_1(c), \epsilon_2(c)$ are  associated complementary  functions.  We claim that
each $\epsilon_j(c)$ is  decreasing, $\epsilon_j(+\infty)=0$ and this  implies the  monotonic character of   $T(c)$. 
To prove the monotonicity of $\epsilon_1(c)$,   consider the interval $[\lambda_1(c_1),\tilde \lambda_1(c_1)]$ for positive  $c_1<c_2$, the parabola $y_1(z)=z^2 -c_1z - 1$ as well as  the shifted parabola $\bar y_2(z)= y_2(z+\alpha),$ $y_2(z)=z^2 -c_2z - 1$,  where  $\alpha >0$ is chosen to comply with $y_1(\lambda_1(c_1))=\bar y_2(\lambda_1(c_1))$. 
It is clear that the graph of $\bar y_2(z)$ is a shifted (horizontally  and downward) copy of the graph of $y_1(z)$. Hence,  $y_1'(z)<\bar y_2'(z)$ for  $z \in [\lambda_1(c_1),  \lambda_1(c_1)]$ 
and we can conclude that parabola $\bar y_2$ intersects the abscissa axis at some point from the interval $(\lambda_1(c_1),\tilde \lambda_1(c_1))$ so that its intersection with the  graph of $y=-ke^{-c_2h(z+\alpha)}$ belongs to the same interval. This means that
$\epsilon_1(c_1) = \tilde \lambda_1(c_1)-\lambda_1(c_1)>\tilde \lambda_1(c_2)-\lambda_1(c_2)=\epsilon_1(c_2)$. The fact that $\epsilon_2(c)$ decreases can be proved in a similar way. The property  $\epsilon_j(+\infty)=0$, $j=1,2$, is geometrically obvious.  

For a fixed $h \geq 0$, the function $T(c)$ is defined for all $c \geq c_\#(h)$ (note that $\lambda_1(c)$ is not defined for $c < c_\#(h)$). As we have seen, 
$T(+\infty) =1 > (3-k)/4$ so that the unique pushed wave to (\ref{MGDt}) exists if and only if  $T(c_\#(h)) \leq (3-k)/4$. Therefore, for each $k \in (1,5/3)$, the critical speed  $c_*(h)$ is well defined by  (\ref{theta}) and satisfies $c_*(h) > c_\#(h)$ on some  maximal interval $h \in [0,h_p),$ with $h_p \in (0, +\infty]$ depending on $k$. 

It is known that $c_\#(h)$ is a decreasing function of $h$ (see e.g. \cite[Lemma 1.2]{LMS}). The minimal speed $c_*(h)$ has the same property when the birth function  $g$ increases between $0$ and $\kappa$, cf. \cite[Lemma 3.5]{LZii}. The next result shows that  the monotonic nature of  $c_*(h)$ is also preserved in our unimodal toy model: 
\begin{lemma}The above defined function $c_*:[0,h_p] \to (0,+\infty)$ is continuous and strictly decreasing. 
\end{lemma}
\begin{proof} Since the quotient $\lambda_1/\mu_1$ is a smooth function of $c,h$ and equation (\ref{theta}) has a unique solution for each $h \in [0,h_p]$, continuity of $c_*$ follows. 
Next, consider the roots $\lambda_1=\lambda_1(c,h)$ and $\mu_1=\mu_1(c,h)$, this time also indicating their dependence on $h$. It is easy to see that for  fixed speed $c$  and delays $h_1<h_2$,  $h_j\in [0,h_p]$, 
$$
\frac{\lambda_1(c,h_1)}{\mu_1(c,h_1)}<\frac{\lambda_1(c,h_2)}{\mu_1(c,h_2)},
$$
so   $\lambda_1/\mu_1$ is  an increasing function of delay as well as velocity. Therefore, since 
$$
\frac{\lambda_1(c_*(h_2),h_2)}{\mu_1(c_*(h_2),h_2)} = \frac{3-k}{4}  =\frac{\lambda_1(c_*(h_1),h_1)}{\mu_1(c_*(h_1),h_1)} < \frac{\lambda_1(c_*(h_1),h_2)}{\mu_1(c_*(h_1),h_2)} ,
$$
we conclude that $c_*(h_2) < c_*(h_1)$. 
\end{proof}
An important question concerns the possibility of the intersection of the graphs of $c_*(h)$ and $c_\#(h)$.  If this happens and   $c_*(0)> c_\#(0)$ but $c_*(h')=c_\#(h')$ for some 
finite $h'>0$, the delay  is transforming minimal pushed fronts into pulled fronts (here we are somewhat freely paraphrasing an expression coined  in \cite{EP}). This fact is important for the dynamics:  it is known  that pushed waves have better stability properties than the pulled minimal wavefronts. To show that this phenomenon is occurring in model (\ref{MGDt}), it is enough to find $k \in (1,5/3)$ and $h >0$ such that  
$$
T_1(h) = \frac{\lambda_1(c_\#(h),h)}{\mu_1(c_\#(h),h)} > \frac{3-k}{4},
$$
clearly, for such a particular $h$, equation  (\ref{theta}) does not have a solution. Here we can use the limit value $
T_1(+\infty) = {\lambda_{\infty}}/{\mu_{\infty}},
$ (cf. \cite[Lemma 1.2]{LMS}), 
where 
$$\lambda_\infty= \sqrt{1+\frac{1}{\rho^2}}- \frac{1}{\rho}, \quad \rho = \sqrt{w_+(2+w_+)}, $$ 
with $w_+$ being the unique positive root of the equation
$$
e^{-w}(2+w)=2/k, 
$$
and $\mu_\infty$ is the unique positive root for 
$$
\mu^2-1=e^{-\mu \rho}. 
$$

In a similar fashion, we can investigate  the intersection of the curves $c=c_*(h)$ and $c=c_\kappa(h)$.  If this happens and   $c_*(0)< c_\kappa(0)$ but $c_*(h')>c_\kappa(h')$ for some 
finite $h'>0$, the delay  is changing the   monotonicity  of pushed fronts. Particularly, the profile of the wavefront  propagating with the minimal speed $c_*(h')$ is slowly oscillating around the positive equilibrium, cf. \cite{TTT}. 
The same situation was  observed for akin bistable wavefronts \cite{ADG,TVN} to (\ref{MGD}). Arguing as in the previous paragraphs, we find that it suffices to indicate  $k \in (1,5/3)$ and $h' >0$ such that  
$$
T_2(h') = \frac{\lambda_1(c_\kappa(h'),h')}{\mu_1(c_\kappa(h'),h')} < \frac{3-k}{4}.
$$
The latter implies $c_\kappa(h') < c_*(h')$.   If the curves $c=c_\kappa(h)$ and $c=c_\#(h)$ do not intersect, we can use the limit value $
T_2(+\infty) = {\hat\lambda_{\infty}}/{\hat\mu_{\infty}},
$ (cf. \cite[Lemma 1.1]{LMS}), 
where 
$\hat\lambda_\infty >0$ and $\hat \mu_\infty <0$ can be found from the  equations
$$
\lambda^2-1+ke^{-\hat\rho\lambda}=0, \quad \mu^2-1=e^{-\mu \hat\rho} 
$$ 
resp., 
where $\hat \rho = \sqrt{w_-(2+w_-)}$
and  $w_-$ is  the unique negative root of the equation
$$
e^{-w}(2+w)=-2. 
$$
In the two examples given below, some numerical and geometrical evidences suggest that $T_1(h)$ is an increasing function, so that there can be at most one point of intersection of the graphs of $c_\#$ and $c_*$. 
\begin{figure}[h]
\vspace{-18mm} 
\begin{center} 
\scalebox{0.32}{\hspace{-25mm} \includegraphics{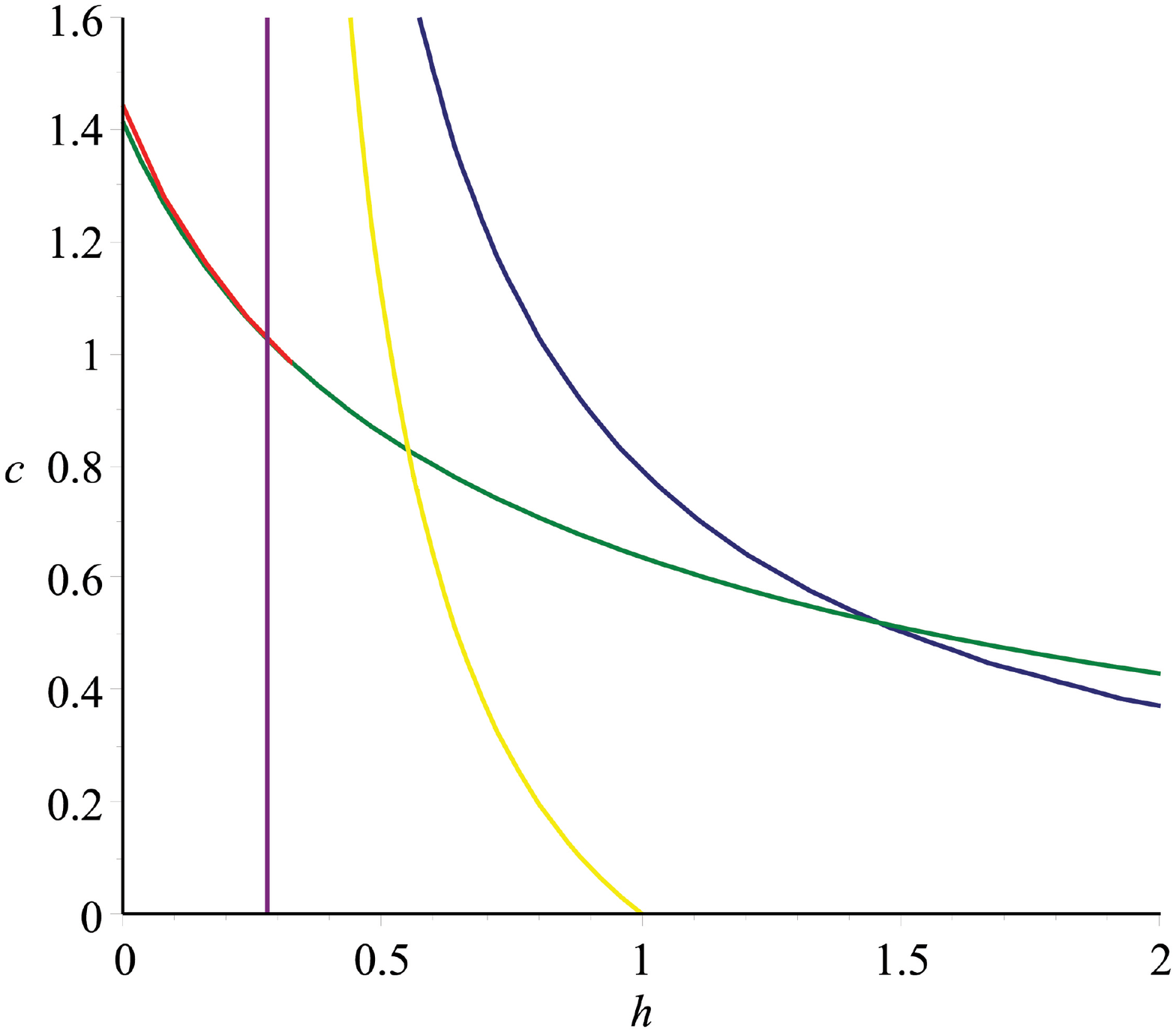} \includegraphics{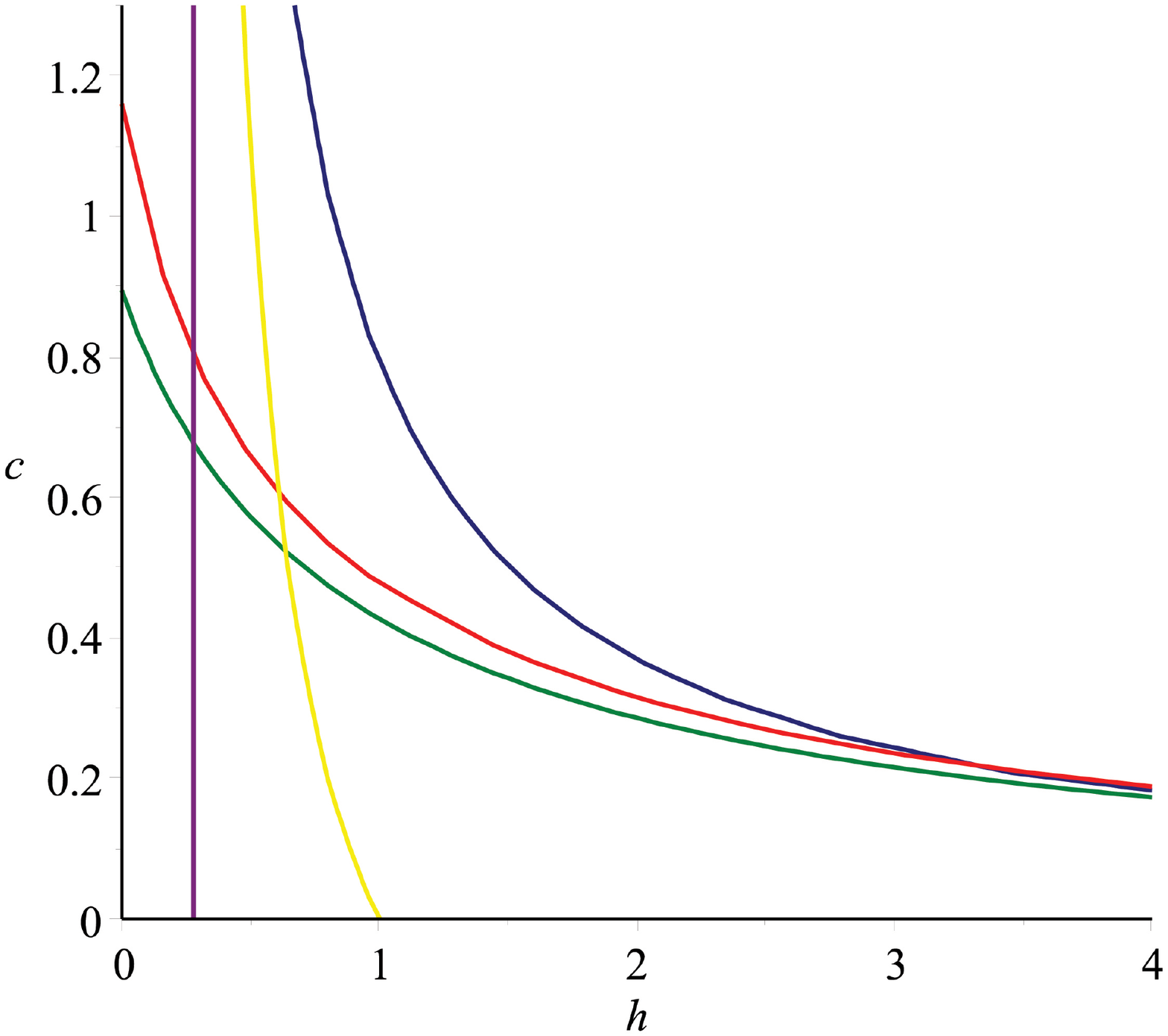}}

\vspace{-18mm}

\caption{Vertical asymptote $h=h_*$ followed (in the counter-clockwise direction)  by the graphs of $c(h), c_\#(h), c_*(h),c_\kappa(h)$. The cases $k=1.5$ (left) and $k=1.2$ (right).}
\label{Fig2}
\end{center}
\end{figure}
Take first  $k=1.5$.  Then $w_+= 0.7088\dots$, $\rho=1.3856\dots,\lambda_\infty=0.5115\dots,$ $\mu_\infty=1.1031\dots...$ and
$\lambda_\infty/\mu_\infty=   0.4637\dots >(3-1.5)/4= 0.375$. In this particular case, the minimal wavefront is pulled for all  $h>0.3379\dots$ and is pushed 
for all  $h<0.3379\dots$, see Figure \ref{Fig2} (left). Next, consider $k=1.2$. Then    $w_+= 0.3388\dots$,  $\rho=0.8901\dots,$ $\lambda_\infty=0.3806\dots,$ $\mu_\infty= 1.1639\dots$,   so that 
$\lambda_\infty/\mu_\infty=   0.3269\dots <(3-1.2)/4= 0.45$. In this particular case, the minimal wavefront is pushed for all  $h\geq 0$, see Figure \ref{Fig2} (right) and compare $T_1(4)=0.3141\dots$ with $T_1(+\infty)= 0.3269\dots$. If additionally  $h >h_{osc} = 3.25\dots$, after crossing the level $u=1$, this pushed wavefront  slowly oscillates \cite{TTT} around it. Here $h_{osc}$ is determined from the equation $T_2(h_{osc}) =(3-k)/4=0.45$, the graphs of $c_\kappa$ and $c_\#$ have an intersection.    Hence, contrarily to the first case (when $k=1.5$), delay is not changing  the pushed nature of minimal wavefronts when we have relatively `strong' non-subtangency of $g$ at $0$ ($g$ is taken with $k=1.2$). 

The following result completes the above discussion by showing that  for every $h \geq 0$ the set $\mathcal C(h)$ of all possible velocities of semi-wavefronts for  model (\ref{MGDt}) has the usual structure of  semi-infinite  interval $[c',+\infty)$. 
\begin{theorem} \label{T53}
Suppose that $k\in (1,3)$ is such that  (\ref{theta}) has solution $c_*\geq  c_\#$.
Then $c_*$ is the minimal speed of propagation in the sense that equation (\ref{MGDt}) has a wavefront solution 
propagating with the speed $c$  if and only if $c \geq c_*$.  Furthermore, $c_\#$ is the minimal speed of propagation if  (\ref{theta}) does not have a  positive root $c$.
\end{theorem}
\begin{proof} Clearly, 
each non-critical (i.e. $c\not=c_\#$)  wavefront profile normalized by the condition $\phi(-ch)=1$ should be of the form  
\begin{equation}\label{ip}
\phi(t)=e^{\lambda_2(t+ch)} p + (1-p) e^{\lambda_1(t+ch)}, \quad t\leq 0,
\end{equation}
with some appropriate $p\geq 0$. Since $\lambda_2 < \lambda_1$, we have that $\phi(t)>0, \ t \leq -ch$. 
Similarly to the case of pushed wavefronts ($p=0$), relation (\ref{laplaceMGDtp2}) yields the following determining equation for the admissible speeds $c >c_\#$: 
$$
\lambda_2 e^{\lambda_2ch}p+ (1-p)\lambda_1 e^{\lambda_1ch}+ (\mu_1-c)(e^{\lambda_2ch} p + (1-p) e^{\lambda_1ch}-2)+ $$
$$e^{-\mu_1 ch}\int\limits_{-ch}^{0}e^{-\mu_1 t}\left(e^{\lambda_2(t+ch)} p + (1-p) e^{\lambda_1(t+ch)}-2\right)dt = 0. 
$$
After a straightforward computation, we  find the following relation between $p, c, k$: 
$$
\frac{1-p}{1-\lambda_1/\mu_1} + \frac{p}{1-\lambda_2/\mu_1}  = \frac{4}{1+k}, 
$$
from which 
\begin{equation}\label{ppp}
p=  \frac{4\lambda_1/\mu_1- (3-k)}{(1+k)}\left(\frac{\mu_1-\lambda_2}{\lambda_1-\lambda_2}\right) \leq \frac{\mu_1-\lambda_2}{\lambda_1-\lambda_2}, 
\end{equation}
where $\lambda_2 <\lambda_1 <\mu_1$ for $c > c_\#$. Formula (\ref{ppp}) shows the uniqueness of the profile $\phi$ normalized by $\phi(-ch) =1$ (equivalently,  the uniqueness of 
$p$) for each fixed admissible $c$. We will write $\phi(t,c)$ to indicate the dependence of $\phi$ on $c$. Since $p$ is a continuous function of $c$, we conclude that $\phi(t,c)$ depends continuously on $t,c$. 

Using (\ref{ppp}), in accordance with \cite{TTT}, we obtain that $\phi'(t,c) >0$ for all $t \leq -ch$. To prove
the latter fact, it suffices to establish the positivity of $\phi'(-ch,c)$: 
$$
(1+k)\phi'(-ch,c) = (3-k)(\mu_1-\lambda_1-\lambda_2) + \frac{4\lambda_1\lambda_2}{\mu_1}> 
$$
$$
\frac{3-k}{\mu_1}(\mu_1^2-(\lambda_1+\lambda_2)\mu_1 + 2\lambda_1\lambda_2)>
 \frac{3-k}{\mu_1}(\lambda_1^2-(\lambda_1+\lambda_2)\lambda_1 + 2\lambda_1\lambda_2)>0. 
$$

Next, it follows from (\ref{ppp}) that   $p \geq 0$ if and only if 
$$
\frac{\lambda_1(c)}{\mu_1(c)}\geq \frac{3-k}{4}. 
$$
Hence,  for each $c \geq c_\#$ satisfying the latter inequality, the initial part $\phi(t,c)$ given by (\ref{ip}) can be continued for $t > ch$ as a solution of 
(\ref{MGDtp}), with $\phi(+\infty,c)=2$. Here we are assuming that $\phi(t,c) > 1$ for all $t > -ch$, this assumption is automatically satisfied when $\phi(t)$ is monotone increasing on $\R$ (i.e. when  $(h,c) \in {\mathcal D}_\kappa$). In the general case, connect two points 
$$(0,c'),  (\tilde h,\tilde c) \in {\mathcal D}_\star = \{(h,c): c\geq c_\#(h)\ \mbox{and} \  \lambda_1(c)/\mu_1(c)\geq (3-k)/4\}. 
$$  with  some continuous path $(h(s),c(s)), \ s \in [0,1]$, lying in  ${\mathcal D}_\star$.  Let $\phi(t,s):= \phi(t,c(s))$ be the corresponding family of profiles, it depends continuously on $t,s$.  
Suppose now that $\phi(t',1) \leq 1$ for some $t' > -ch$. Since $\phi'(t,0)>0, \ t \in \R$, and $\phi(+\infty,s) =2$, there exist the smallest value of $s_1$ and some $t_1 > -ch$ 
such that $\phi(t_1,s_1)=1$ and $\phi'(t_1,s_1)=0,\  \phi''(t_1,s_1)\geq 0$.   This actually implies that $t_1>0$ since otherwise $\phi(t,s_1)$ has two critical points on 
$(-\infty,0]$ which is not possible in view of (\ref{ip}).  But then  (\ref{MGDtp}) and Lemma \ref{le3} lead to the following contradiction: 
$$
0=\phi''(t_1,s_1)-c(s_1)\phi'(t_1,s_1)-\phi(t_1,s_1) +4 - \phi(t_1-c(s_1)h(s_1),s_1) > -1 +4 - 3=0. 
$$
Hence, $\phi(t,s) > 1$ for all $t > -ch$, $s \in [0,1]$ that legitimizes the construction of  a wavefront to (\ref{MGDp}) by $C^1-$continuously gluing its initial part (\ref{ip}) with the unique  piece of the corresponding solution for equation (\ref{MGDtp}). 
This finalizes the proof of Theorem \ref{T53}.   
\end{proof}
On Figure 3, we present solution $u(t,x)$ of the Cauchy problem  
\begin{equation}\label{CP}
u_0(t,x)= \left\{ \begin{array}{ll}
0, & \mbox{as} \,\,\, x<0, \enspace t \in [-h,0], \\
2, & \mbox{as} \,\,\,  x \geq 0, \enspace t \in [-h,0],
\end{array}
\right.
\end{equation}
to equation  (\ref{MGDt}) at the indicated sequence $t=t_j$ of moments.  The numerical simulations are based on the Crank-Nicholson method which is second-order accurate in both spatial and temporal directions. The spatial step size is chosen as $\Delta x=0.05$ in the computational interval $x \in [-25,25]$ together with the Dirichlet boundary conditions $u(t,-25)=0$ and $u(t,25)=2$. The temporal step size is $\Delta t=0.01$.

It is known from \cite{WNH} that in the monotone case (i.e. when the birth function $g$  is increasing on $[0,\kappa]$), solution $u(t,x)$ of (\ref{CP}) exponentially rapidly converges 
to the pushed wavefront.  In sight of the above developed theory  of equation  (\ref{MGDt}), it is natural to expect that  its solution $u(t,x)$  also will converge to the pushed wavefront. Comparison of our theoretical and numerical results corroborates this fact.  For example, on Figure 3 (left) we present snapshots of solution $u(t,x)$ already stabilized around the pushed wavefront 
 for model (\ref{MGDt}) considered with $h=0.5$ and $k=1.2$.  By our theory, in this case the profile of pushed wave is monotone. The right part of Figure 3 shows a magnified fragment of the  leading edge of pushed wave for parameters $h=6$ and $k=1.2$. Our theory predicts that  the profile of  wavefront must be non-monotone in this case.  In good accordance with the theory,  in this case, the numerics presents a profile of pushed wave oscillating around the positive equilibrium (the amplitude of oscillations is relatively small). 

\begin{figure}[h]
\vspace{-52mm} 
\begin{center} \hspace{-10mm}
\scalebox{0.6}{\includegraphics{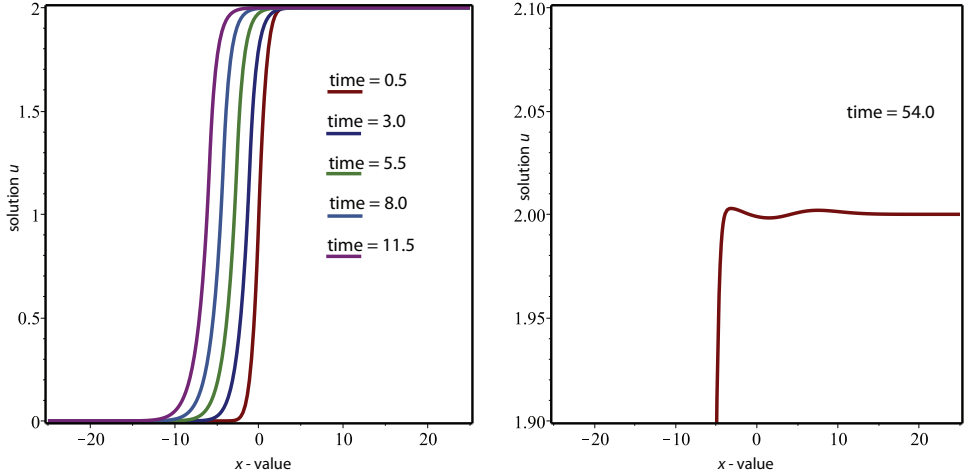}}
\vspace{-53mm} 
\caption{Snapshots of solution $u(t,x)$ to the Cauchy problem (\ref{CP}), $k=1.2$, converging to the pushed wavefront, at the indicated sequence of times $t=t_j$. The cases $h=0.5$ (left), $h=6$ (right).}
\label{Fig3}
\end{center}
\end{figure}

\vspace{-2mm} 

Furthermore, we have compared the theoretical values $c_\#(h), c_*(h)$ and the speeds  $c_{ns}(h)$ obtained from the numerical simulations  for different  values of delay $h$:

\vspace{3mm}

\mbox{\hspace{-5mm}  
\hspace{14mm} 
\begin{tabular}{|cccc||cccc|}
\hline 
 $h$  & $c_\#(h)$ & $c_*(h)$ & $c_{ns}(h)$ &$h$  & $c_\#(h)$ & $c_*(h)$ & $c_{ns}(h)$ \\
 \hline 
 0.5 & 0,5720 & 0,6562 & 0,6377 & 3.5 & 0,1922 & 0,2091 & 0,2112\\
 1 & 0,4270 & 0,4770 & 0,4662 &  4 & 0,1733 & 0,1883 & 0,1892\\
 1.5 & 0,3420 & 0,3779 & 0,3746 & 4.5 & 0,1579 & 0,1713 & 0,1727\\
 2 & 0,2860 & 0,3138 & 0,3165 & 5 & 0,1450 & 0,1571 & 0,1572\\
 2.5 & 0,2458 & 0,2687 & 0,2688 &  5.5 & 0,1340 & 0,1452 & 0,1461\\
 3 & 0,2157 & 0,2351 & 0,2353 & 6 & 0,1246 & 0,1348 & 0,1346\\
 \hline
 \end{tabular}}
 
\vspace{5mm} 

In the table, we can observe a good agreement between theoretical and numerical  values of the minimal speed for monostable wavefronts. Clearly, the minimal wave has  pushed character. In any event, as seen from a point of  view of  rigorous analytical proofs, such a convergence of solution (\ref{CP})   to a pushed wavefront remains  a difficult open problem.

\section{Appendix}
As it was established in  \cite[Lemma A.2]{TVN} and  \cite[Lemma 1.1]{LMS}, for each pair $(h,c) \in {\mathcal D}_\kappa$,  the characteristic function  $\chi_\kappa$ has exactly three real zeros, one positive and two negative (counting multiplicity), $\mu_3 \leq \mu_2 <0 < \mu_1$.  In addition, every complex zero $\mu_j$ of $\chi_\kappa$ is simple \cite[Lemma A.2]{TVN} and has its real part $\Re \mu_j < \mu_2$  \cite[Lemma 1.1]{LMS}.  We claim that actually $\Re \mu_j < \mu_3$ for each complex zero $\mu_j$ of $\chi_\kappa$. Indeed, fix $c=\bar c$ and consider the unique value $h_\star>0$ such that $(h_\star, \bar c)$ belongs to the boundary of $ {\mathcal D}_\kappa$ (i.e. $c_\kappa(h_\star)=\bar c$, cf. Section \ref{fund}).  Our claim is trivially valid for the parameters $(h,c)= (h_\star, \bar c)$ because of  $\mu_2=\mu_3$.  Moreover, since each half-plane $\Re z \geq \alpha$ 
contains at most finite number of zeros of $\chi_\kappa$ and $\mu_j(h)\not \in \R$, being a simple complex zero, depends continuously on $h>0$,  the above claim is also 
valid for all $h$ from some maximal  interval $(h_\delta, h_\star]$. If $h_\delta >0$ then $\Re\mu_k(h_\delta) =\mu_3$ for some index $k$. In this way,  there are at least 
three zeros of $\chi_\kappa$ having the same real part. However, by \cite[Lemma A.2]{TVN}, this situation cannot occur. 

\section*{Acknowledgments} \noindent 
The work of Karel Has\'ik, Jana Kopfov\'a and Petra N\'ab\v{e}lkov\'a was supported  by the institutional support
for the development of research organizations I\v CO 47813059.
  S. Trofimchuk  was   supported by FONDECYT (Chile),   project 1190712. 

\bibliographystyle{amsplain}

\end{document}